\date{Rev. 22/XI/13 JM}
\title{Curves in $\R^d$ intersecting every hyperplane
at most $d+1$ times
}
\newif\ifcmts
\cmtstrue

\newcommand{\cmt}[1]{\ifhmode\newline\fi{\sf *** \ \ #1 \\}}

\author{
{\sc Imre B\'ar\'any}\\
  {\footnotesize R\'enyi Institute of Mathematics}\\[-1.5mm]
  {\footnotesize Hungarian Academy of Sciences}\\[-1.5mm]
  {\footnotesize POBox 127, 1364 Budapest, Hungary, 
   and}\\
  {\footnotesize Department of Mathematics}\\[-1.5mm]
  {\footnotesize University College London}\\[-1.5mm]
  {\footnotesize Gower Street, London WC1E 6BT, England}
\\[-1.5mm]   {\footnotesize e-mail: {\tt barany@math-inst.hu}}
                        \and
{\sc Ji\v{r}\'{\i} Matou\v{s}ek}
\\
   {\footnotesize Department of Applied Mathematics}\\[-1.5mm]
   {\footnotesize  Charles University, Malostransk\'{e} n\'{a}m. 25}\\[-1.5mm]
{\footnotesize  118~00~~Praha~1,
  Czech Republic, and}\\
{\footnotesize    Institute of  Theoretical Computer Science}\\[-1.5mm]
{\footnotesize    ETH Zurich,
      8092 Zurich, Switzerland}
\\[-1.5mm]   {\footnotesize e-mail: {\tt matousek@kam.mff.cuni.cz}}
\and
{\sc Attila P\'or}\\[-1.5mm]
{\footnotesize Department of Mathematics}\\[-1.5mm]
{\footnotesize Western Kentucky University}\\[-1.5mm]
{\footnotesize  1906 College Heights Blvd. \#11078}\\[-1.5mm]
{\footnotesize   Bowling Green, KY 42101, USA}
\\[-1.5mm]   {\footnotesize e-mail: {\tt  attila.por@wku.edu}}
}

\documentclass[11pt]{article}

\usepackage{a4wide}
\usepackage{latexsym}
\usepackage{amssymb,amsmath,amsthm}
\usepackage{graphicx}
\usepackage{url}

\newtheorem{theorem}{Theorem}[section]

\newtheorem{lemma}[theorem]{Lemma}

\newtheorem{claim}[theorem]{Claim}
\newtheorem{corollary}[theorem]{Corollary}

\newcommand{\heading}[1]{\vspace{1ex}\par\noindent{\bf\boldmath #1}}

\makeatletter
\newcommand{\ProofEndBox}{{\ifhmode\unskip\nobreak\hfil\penalty50 \else
          \leavevmode\fi\quad\vadjust{}\nobreak\hfill$\Box$
            \finalhyphendemerits=0 \par}}
\makeatother

\newcommand{\R}{{\mathbb{R}}}

\newcommand\eps{\varepsilon}

\newcommand{\sgn}{\mathop {\rm sgn}\nolimits}
\newcommand{\aff}{\mathop {\rm aff}\nolimits}

\newcommand{\dist}{\mathop {\rm dist}\nolimits}
\numberwithin{equation}{section}

\providecommand\RR{\mathcal{R}}

\def\:{\colon}

\DeclareMathOperator{\OT}{OT}
\newcommand{\COT}{\OT^*}
\DeclareMathOperator{\twr}{twr}
\newcommand\crossi[1]{$({\le}\,#1)$-crossing}

\newcommand\al{\alpha}

\newcommand\de{\delta}
\newcommand\si{\sigma}

\long\def\onefigure#1#2{
\begin{figure*}[tbp]
\begin{center}
#1
\end{center}
\caption{#2}
\end{figure*}
}

\def\immediateFigure#1{%
\smallskip\begin{center}#1\end{center}\smallskip }

\newcommand{\labfig}[2]  
{\onefigure{\mbox{\includegraphics{#1}}}{\label{f:#1} #2} }

\newcommand{\labfigw}[3]  
{\onefigure{\mbox{\includegraphics[width=#2]{#1}}}{\label{f:#1} #3}}

\newcommand{\immfig}[1]  
{\immediateFigure{\mbox{\includegraphics{#1}}}}

\newcommand{\immfigw}[2] 
{\immediateFigure{\mbox{\includegraphics[width=#2]{#1}}}}

\ifcmts
\newcommand{\marrow}{\marginpar{\boldmath$\longleftarrow$}}
\newcommand{\jirka}[1]{\ifhmode\newline\fi\marrow \textsf{*** (JIRKA: ) #1\newline}}
\newcommand{\imre}[1]{\ifhmode\newline\fi\marrow \textsf{*** (IMRE: ) #1\newline}}

\else
\newcommand{\marrow}{}
\newcommand{\jirka}[1]{}
\newcommand{\imre}[1]{}

\fi

\begin{document}

\maketitle

\begin{abstract} By a curve in $\R^d$ we mean a continuous map
$\gamma\:I\to\R^d$, where $I\subset\R$ is a closed interval.
 We call a curve $\gamma$ in $\R^d$ {\crossi{k}}
if it intersects every hyperplane at most $k$ times
(counted with multiplicity). The {\crossi{d}} curves in $\R^d$
are often called \emph{convex curves} and they
form an important class; a primary example is the \emph{moment curve}
$\{(t,t^2,\ldots,t^d):t\in[0,1]\}$. They are also closely related
to \emph{Chebyshev systems}, which is a notion of considerable
importance, e.g., in  approximation theory. Our main result is that
for every $d$ there is $M=M(d)$ such that every {\crossi{d+1}} curve in $\R^d$
can be subdivided into at most $M$ {\crossi{d}} curve
segments. As a consequence, based on the work of Eli\'a\v{s}, Rold\'an, Safernov\'a,
and the second author, we obtain an essentially tight
lower bound for a geometric Ramsey-type problem in $\R^d$
concerning order-type homogeneous sequences of points, investigated
in several previous papers.
\end{abstract}

\section{Introduction}

The most intuitive statement of the problem investigated in this paper
involves curves in $\R^d$. By a curve we mean an arbitrary
continuous mapping $\gamma\:I\to\R^d$, where $I\subset\R$ is a closed interval
(we could admit an open interval as well, but this would add
unnecessary technical complications).
Let us say that a curve $\gamma$ in $\R^d$ is \emph{\crossi{k}}
if it intersects every hyperplane $h$ at most $k$
times.\footnote{For algebraic curves in the complex projective space,
the number of intersections with a
generic hyperplane is the \emph{degree}, but we prefer using
a different term, since we deal with much more general curves,
which are typically not algebraic.}
Here the intersections
are counted with multiplicity; that is, the condition of {\crossi{k}}
reads $|\{t\in I: \gamma(t)\in h\}|\le k$.

It will be useful to observe that a {\crossi k} curve is not
constant on any nonempty open interval, and its image contains no segment.

\heading{\boldmath {\crossi{d}} ($=$convex) curves. }
The {\crossi{d}} curves in $\R^d$ are called \emph{convex curves}
in a significant part of the literature (e.g.,
\cite{arnold-problems,z-tm-04,shapishapi,sedykh-shapi,musin-cheby}),
and they are of considerable interest
in several areas.
In the plane, a convex curve in this sense is a connected piece of the
boundary of a convex set.
A primary example  of a higher-dimensional convex curve
is the \emph{moment curve} $\{(t,t^2,\ldots,t^d):t\in[0,1]\}$. The convex
hull of $n\ge d+1$ points on a convex curve in $\R^d$
is  a \emph{cyclic
polytope}, one of the most important examples in the theory of convex
polytopes and in discrete geometry in general.

If we regard a convex curve $\gamma\:I\to\R^d$ as a $d$-tuple
$(\gamma_1,\ldots,\gamma_d)$ of functions $I\to\R$, and define
$\gamma_0\equiv 1$, then the $(d+1)$-tuple $(\gamma_0,\gamma_1,\ldots,
\gamma_d)$ (or possibly $(-\gamma_0,\gamma_1,\ldots,
\gamma_d)$) forms a \emph{Chebyshev system},\footnote{Let $A$ be a linearly ordered
set of at least $k+1$ elements. A (real) \emph{Chebyshev system} on $A$
is a system of continuous real functions
 $f_0,f_1,\ldots,f_k\:A\to\R$ such that for every choice of elements
$t_0<t_1<\cdots<t_k$ in $A$, the matrix $(f_i(t_j))_{i,j=0}^k$ has
a (strictly) positive determinant.} which is an important notion
in approximation theory, theory of finite moments, and other areas---see,
e.g., \cite{KarlinStudden,CPZ}. Conversely, every Chebyshev system
$(\gamma_0,\ldots,\gamma_d)$ on an interval $I$ with $\gamma_0\equiv 1$
(or more generally, $\gamma_0$ strictly monotone) gives rise to
a convex curve in~$\R^d$.

\heading{Subdividing \boldmath {\crossi{d+1}} curves. }
The following question is quite natural and interesting in its own right and it
has been motivated by the work \cite{EMRS} in geometric Ramsey theory,
as will be explained below. Given an integer $d\ge 2$, does there exist $M=M(d)$
such that every {\crossi{d+1}} curve $\gamma$ in $\R^d$ can be subdivided
into at most $M$ convex curves?  In more detail,
if $\gamma$ is a map $I\to\R^d$, we want to subdivide $I$ into
subintervals $I_1,\ldots,I_k$, $k\le M$, so that the restriction
of $\gamma$ to each $I_i$ is convex (i.e., {\crossi{d}}).
Our main result answers this question in the affirmative.

\begin{theorem}\label{t:main}
For every integer $d\ge 2$ there exists $M=M(d)$
such that every {\crossi{d+1}} curve $\gamma$ in $\R^d$ can be subdivided
into at most $M$ convex curves.
\end{theorem}

We note that the value $d+1$ is important, since a  {\crossi{d+2}}
curve in $\R^d$ in general cannot be subdivided into a bounded number
of convex curves. An example for $d=2$ can be obtained, e.g.,
by starting with a circular arc and making many very small and flat inward
dents in it.


The case $d=2$ is already nontrivial, but to our surprise,
we haven't found it mentioned in the literature. The following
picture shows a planar curve, namely, the graph of
$x(1-x^2)^2$ on $[-1,1]$, which can be checked to be  {\crossi 3},
but obviously cannot be subdivided into fewer than $4$ convex arcs:
\immfigw{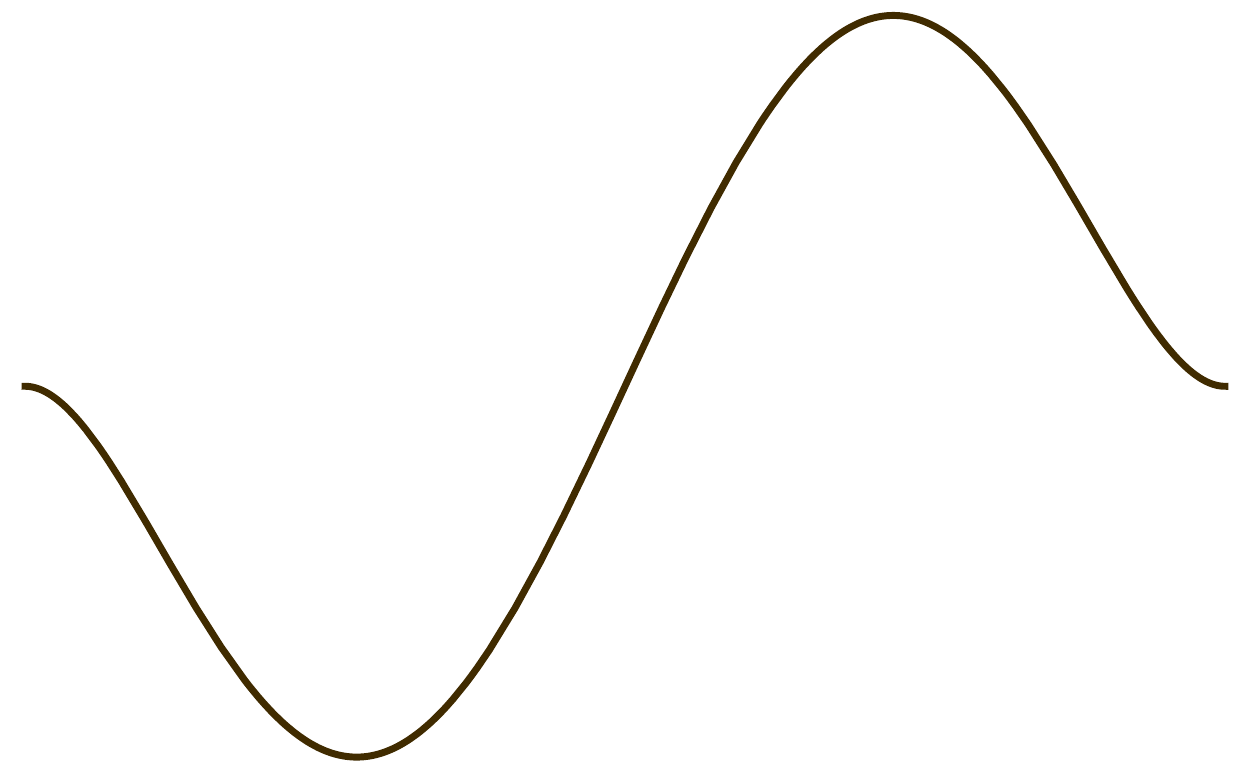}{5cm}
Hence $M(2)\ge 4$. We can prove that $M(2)$ actually equals $4$, and that $M(3) \le 22$.
The proofs can be found in
an earlier version of this paper \cite{curve-bm} by the first two authors.

\heading{Theorem~\ref{t:main} for polygonal paths.} For technical reasons,
and also from the point of view of our motivation in
geometric Ramsey theory, it is more convenient to work with polygonal
paths. A \emph{polygonal path}
 is a curve made of finitely many straight segments;
we call these segments the  \emph{edges} of the polygonal path, and
their endpoints are the \emph{vertices}. For a point sequence
$(p_1,p_2,\ldots,p_n)$, we write $p_1p_2\cdots p_n$ for the polygonal
path consisting of the segments $p_1p_2$,\ldots, $p_{n-1}p_n$.

The definition of \crossi{k}\  needs to be modified: we call
a polygonal path $\pi$ \emph{\crossi{k}} if it intersects every
hyperplane in at most $k$ points, \emph{with the exception of
the hyperplanes that contain an edge of $\pi$}.
Moreover, we will also consider only \emph{polygonal paths in general position},
meaning that every $k\le d+1$ vertices of the polygonal path
are affinely independent. The polygonal path version of Theorem~\ref{t:main}
says the following.

\begin{theorem}\label{t:pmain}
For every integer $d\ge 2$ there exists $M=M(d)$
such that every {\crossi{d+1}} polygonal path $\pi$ in $\R^d$ can be subdivided
into at most $M$ convex (i.e., {\crossi{d}}) polygonal paths.
\end{theorem}

In Section~\ref{s:poly-c} we prove by a limit argument
that Theorem~\ref{t:pmain} implies Theorem~\ref{t:main}.


\heading{Order-type homogeneous subsequences. } Now we come to the
geometric Ramsey-type problem motivating our work.

Let $T=(p_1,\ldots,p_{d+1})$ be an ordered $(d+1)$-tuple of points
in $\R^d$. We recall that the \emph{sign} (or \emph{orientation})
of $T$ is
defined as $\sgn\det X$, where the $j$th column of the $(d+1)\times (d+1)$
matrix $X$ is $(1,p_{j,1},p_{j,2},\ldots,p_{j,d})$, with $p_{j,i}$ denoting
the $i$th coordinate of $p_j$. Geometrically, the sign
is $+1$ if the $d$-tuple of vectors $p_1-p_{d+1},\ldots,p_{d}-p_{d+1}$
forms a positively oriented basis of $\R^d$, it is $-1$
if it forms a negatively oriented basis, and it is $0$
if these vectors are linearly dependent.

We call a sequence $(p_1,p_2,\ldots,p_n)$
of points in $\R^d$ in general position \emph{order-type homogeneous}
if all $(d+1)$-tuples $(p_{i_1},\ldots,p_{i_{d+1}})$,
$i_1<\cdots<i_{d+1}$, have the same sign (which is nonzero,
by the general position assumption).

Let $\OT_d(n)$ be the smallest $N$ such that every sequence of
$N$ points in general position in $\R^d$ contains an order-type homogeneous
subsequence of length $n$. The existence of $\OT_d(n)$ for all $d$ and $n$
follows immediately from Ramsey's theorem, but several recent papers
\cite{highes-aim,cfpss-semialg,Suk-OT,EMRS}
considered the order of magnitude of $\OT_d(n)$, for $d$ fixed
and $n$ large.

For $d=2$, the classical paper of Erd\H{o}s and Szekeres
\cite{es-cpg-35} implies that $\OT_2(n)=2^{\Theta(n)}$.\footnote{We
employ the usual asymptotic notation
for comparing functions: $f(n)=O(g(n))$ means that
$|f(n)|\le C|g(n)|$ for some $C$ and all $n$, where $C$ may depend
on parameters declared as constants (in our case on $d$);
$f(n)=\Omega(g(n))$ is equivalent to $g(n)=O(f(n))$;
and $f(n)=\Theta(g(n))$ means that both $f(n)=O(g(n))$
and $f(n)=\Omega(g(n))$.}
Suk \cite{Suk-OT}, improving on a somewhat weaker bound
by Conlon et al.~\cite{cfpss-semialg}, proved the upper bound
$\OT_d(n)\le\twr_d(O(n))$ for every fixed $d$, where
the tower function $\twr_k(x)$ is defined by $\twr_1(x) = x$
 and $\twr_{i+1} (x) = 2^{\twr_i (x)}$. He conjectured this
to be optimal, but so far matching lower bounds were known
only for $d=2$ (by \cite{es-cpg-35}) and $d=3$ \cite{highes-aim}.

By combining the results of \cite{EMRS} with Theorem~\ref{t:pmain},
we obtain a matching lower bound for all $d\ge 2$:

\begin{theorem}\label{t:OT4}
We have $\OT_d(n)\ge \twr_d(\Omega(n))$.
\end{theorem}

The argument is given in Section~\ref{s:OT}.

\section{Order-type homogeneity and path convexity }

We need the following fact.

\begin{lemma}\label{f:act}
A sequence $P=(p_1,p_2,\ldots,p_n)$ in general position in $\R^d$
is order-type homogeneous iff the polygonal path
$\pi=p_1p_2\cdots p_n$ is convex.
\end{lemma}

\begin{proof}
First we assume that $P$ is not order-type homogeneous. Then
it has two $(d+1)$-tuples, of the form $Q=(q_1,\ldots,q_{d+1})$
and $R=(r_1,\ldots,r_{d+1})$,
with opposite signs (both $Q$ and $R$ are subsequences of $P$,
i.e., the $q_i$ and the $r_j$ appear in $P$ in this order).

It is easy to check that we can also find $Q$ and $R$
with opposite signs that differ in a single point;
more precisely, there is an index $k$ such that
$q_i=r_i$ for all $i\ne k$. Indeed, given arbitrary $Q$ and
$R$ with opposite  signs, we can convert $Q$ into $R$ by a sequence
of moves, each of them changing a single element: we always move the
first element in which the current $Q$ differs from $R$ to the correct
position. Then at least one of the moves involves two $(d+1)$-tuples
with opposite signs.

Having $Q$ and $R$ as above with $q_i=r_i$ for all $i\ne k$,
we consider the hyperplane $h$ spanned by the points of $Q':=\{q_i:i\ne k\}$.
Then $q_k$ and $r_k$ lie on opposite sides
of $h$, and hence $\pi$ intersects $h$ between $q_k$ and $r_k$.
Together with the $d$ points $Q'$, we have
$d+1$ intersections of $\pi$ with $h$.

This $h$ may still contain edges of $\pi$, so we may need to move
it slightly. For simpler description, we think of $h$ as horizontal,
and say that $q_k$ is below $h$, $r_k$ is above $h$, and $q_k$ precedes
$r_k$ in $P$. Then, since $Q'$ is affinely independent, we can move $h$
by an arbitrarily small amount to a new position $h'$ so that the points
in the sequence $(q_1,q_2,\ldots,q_{k-1},q_k,r_k,q_{k+1},\ldots,q_{d+1})$
are alternatingly above and below $h'$. This implies that $\pi$
intersects $h'$ at least $d+1$ times, and since the move of $h$
was generic, we may assume that $h'$ contains no edges of~$\pi$.

For the reverse implication, we need the following claim:
\emph{If $P=(p_1,p_2,\ldots,p_n)$ is an order-type homogeneous
 sequence and $q$ is an interior point of the segment
$p_ip_{i+1}$, then the sequence $P'=(p_1,p_2,\ldots,p_i,q,p_{i+2},
\ldots,p_n)$ ($p_{i+1}$ replaced with $q$)
is order-type homogeneous as well.}

To verify this claim, we suppose w.l.o.g.\ that
all $(d+1)$-tuples of $P$ are positive,
and we consider an arbitrary $(d+1)$-tuple in $P'$
involving $q$, of the form
\[
T=(p_{j_1},\ldots,p_{j_{k-1}},q,p_{j_{k+1}},\ldots,p_{j_{d+1}}),\
1\le j_1<\cdots<j_{k-1}<i+1<j_{k+1}<\cdots<j_{d+1}\le n.
\]
We think of $q$ moving from $p_i$ to $p_{i+1}$ along the
segment $p_ip_{i+1}$.
The determinant whose sign defines the sign of $T$ is an
affine function of $q$ (considering the remaining points of $T$
fixed).  For $q=p_i$ it is either $0$ (if $j_{k-1}=i$)
or strictly positive, and for $q=p_{i+1}$
it is strictly positive. Therefore, for $q$ in between,
it is strictly positive too, which proves the claim.

Now  we assume for contradiction that the sequence $P=(p_1,\ldots,p_n)$
is order-type homogeneous, but the corresponding polygonal path $\pi$ is not
convex, and so it has at least $d+1$ intersections with
some hyperplane $h$ not containing an edge of $\pi$.
Let us fix intersections $q_1,q_2,\ldots,q_{d+1}$;
at least one of them, call it $q_\ell$, is an interior
point of an edge $p_jp_{j+1}$ of $\pi$ (since the $p_i$ are in
general position).

Using the claim above, we now want to replace $\pi$ by another
polygonal path $\pi'$,
whose vertex sequence is still order-type homogeneous and
includes all $q_i$ with $i\ne \ell$, as well as $p_j$ and $p_{j+1}$.
To this end, we first observe that no two $q_i$ share a segment
of $\pi$ (since $h$ contains no such segment).

When producing $\pi'$, first, if there is a $q_i$ with $i>\ell$
that is not a vertex of the current polygonal path, we take the
last such $q_i$. We replace the vertex of the current polygonal path
immediately following $q_i$ with $q_i$. By the claim, the new vertex sequence
is still order-type homogeneous. We repeat this step until
all $q_i$ with $i>\ell$ become vertices.

Then we proceed analogously with the $q_i$, $i<\ell$, that are
not vertices. This time we start with the smallest $i$,
and $q_i$ always replaces
the vertex immediately preceding it (and we
apply the claim to the reversal of the considered sequences).
Here is an illustration:
\immfig{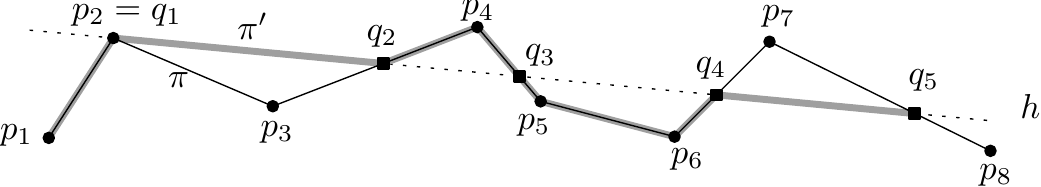}

In this way, we obtain the polygonal path $\pi'$ with order-type
homogeneous vertex sequence that is intersected by the hyperplane
$h$ in the $d$ vertices $q_i$, $i\ne\ell$, and in $q_\ell$,
which is an interior point of the segment $p_jp_{j+1}$
(neither $p_j$ nor $p_{j+1}$ have been replaced).
 But then the $(d+1)$-tuples
$(q_1,\ldots,q_{\ell-1},p_j,q_{\ell+1},\ldots,q_{d+1})$
and $(q_1,\ldots,q_{\ell-1},p_{j+1},q_{\ell+1},\ldots,q_{d+1})$
have opposite signs---a contradiction.
\end{proof}

\section{A combinatorial property of \boldmath{\crossi{d+1}} paths}

Here we prove a combinatorial property of point sequences
in $\R^d$ for which the corresponding polygonal path is
{\crossi{d+1}}. In the two subsequent sections we will
derive Theorem~\ref{t:pmain} from this property
in a purely combinatorial way.

Let $P=(p_1,\ldots,p_n)$ be a sequence  in general position in $\R^d$ and
let $\pi=p_1\cdots p_n$ be the corresponding polygonal path.
For notational convenience,
for $Q\subset P$ with $|Q|=d+1$, we define $\sgn Q$ as the sign of
the sequence $(p_{i_1},\ldots,p_{i_{d+1}})$,
 where $Q=\{p_{i_1},\ldots,p_{i_{d+1}}\}$
with $i_1<i_2<\ldots <i_{d+1}$. For a fixed subset $R\subset P$
with $|R|=d$, we consider the following sequence,
which we call the \emph{sign sequence of $R$}:
\begin{equation}\label{eq:flip}
\Bigl(\sgn(\{p_i\}\cup R):
i=1,2,\ldots,n,\ p_i\not\in R\Bigr)\in \{-1,+1\}^{n-d}.
\end{equation}

\begin{lemma}\label{l:flip} If $\pi$ is {\crossi{d+1}},
then for every $R$ as above, the sign sequence (\ref{eq:flip})
of $R$ has at most one sign change.
\end{lemma}

\heading{A simple case. } For proving the lemma, we first
consider a simple special case. 
Letting $H$ be the hyperplane spanned by $R$,
we  assume that $R$ contains no consecutive elements from $P$, and
moreover, that $H$ separates 
$p_{i-1}$ from  $p_{i+1}$ whenever $p_i \in R$.

Because of the {\crossi{d+1}} condition,
 $(\pi\cap H)\setminus R$ is either the empty set or a single point, 
which we call~$q$.
Then for $x \in \pi$, we have $\sgn(\{x\} \cup R)=0$
iff $x \in R$ or $x=q$. 

Let us think of $x$ moving along $\pi$.
When it passes through a point $p \in R$, $\sgn(\{x\}\cup R)$
does not change because $x$ moves from one side of $H$ to the other, while
$x$ changes places with $p$ in the order on $\pi$.
The same argument shows that $\sgn(\{x\}\cup R)$
changes  only if $x$ passes through~$q$.

\heading{Auxiliary claims. } Next, we make preparations for 
proving the lemma in general.

The set $P\setminus R$ is non-empty, so we fix one of its elements
and call it $p_{\al}$. We define $\RR_{\de}$ as the set of all sequences 
$(q_i\in \pi: p_i \in R)$ such that 
$|q_i-p_i|<\de$, 
and for $i>\al$, $q_i$ lies on the open segment $(p_{i-1},p_i)$,
while for $i<\al$ it lies on  $(p_i,p_{i+1})$. Here is
a schematic illustration:
\immfig{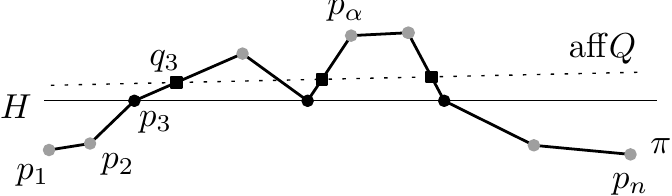}

Since $R$ spans the hyperplane $H$, every set $Q\in\RR_{\de}$ for sufficiently
small $\delta$ spans a hyperplane as well. By general position,
we have $\eps_0:=\dist(P\setminus R,H)>0$.  By continuity,
we also get the next claim:

\begin{claim}\label{cl:dist} There is $\delta_1>0$ such that
$\dist(P\setminus R,\aff Q)>\frac 12 \eps_0$ for all
$Q\in \RR_{\de_1}$.
\end{claim}

This has the following consequence:

\begin{corollary}\label{cor:qqq} 
If $p_h,p_{h+1}\notin R$ and $H \cap p_hp_{h+1}\ne \emptyset$, then  $\aff Q \cap p_hp_{h+1}\ne \emptyset$ for all $Q \in \RR_{\de_1}$.
\end{corollary}

\begin{claim}\label{cl:empty} There is a $\de_2 \in (0,\de_1)$ such that $P\cap \aff Q=\emptyset$ for all $Q \in \RR_{\de_2}$.
\end{claim}

\begin{proof}
If not, then there is a sequence $\de_m \to 0$ and $Q_m \in \RR_{\de_{m}}$ with $P \cap \aff Q_m\ne \emptyset$. Then, for a suitable subsequence,
 $P \cap \aff Q_m$ contains a fixed element $p_h \in P$. 
We have $p_h \in R$ because the $Q_m$ have distance at least
$\eps_0/2$ to $P\setminus R$.

Let $(p_i,p_{i+1},\ldots,p_j)$ be the \emph{string} of $R$ containing
$p_h$, i.e., a maximal contiguous subsequence of $P$ whose points all lie
in $R$ (i.e., $p_{i-1},p_{j+1}\notin R$;
we also admit $i=1$ and $j=n$, as well as $i=j$). 
Thus $i\le h\le j$ and the polygonal path $p_i\ldots p_j$ 
is contained in~$H$.

 Let us assume $h>\al$; then $i>\al$ as well. 
Since $p_h \in \aff Q_m$ and $q_h \in Q_m$, the whole line $\aff \{p_h,q_h\}$
is contained in $\aff Q_m$. Since $p_{h-1}$ is on this line, it is in 
$\aff Q_m$ as well. This shows (by induction) that $p_h,p_{h-1},\ldots,p_i,p_{i-1} \in \aff Q_m$. Thus $p_{i-1}\in \aff Q_m$, which contradicts
 Claim~\ref{cl:dist}. The argument for $h<\al$ is symmetric.
%
\end{proof}

\begin{proof}[Proof of Lemma~\ref{l:flip}]
 We fix some $\de \in (0,\de_2)$ and $Q\in  \RR_{\de}$, and set $H^*=\aff Q$. 
We observe that $H$ and $H^*$ separate the points of $P \setminus R$ the same way. Moreover, if $(p_i,\ldots,p_j)$ is a string of $R$ and $i>\al$, 
then the points $p_{i-1},p_i,\ldots,p_j$ lie alternately on the two sides of $H^*$. This follows from the fact that the path $p_{i-1}p_i\ldots p_j$ intersects $H^*$ in the points $q_i,\ldots,q_j$. Similarly, for $i<\al$,
the points $p_i,\ldots,p_j,p_{j+1}$ lie alternately on the two sides of~$H^*$.

We again let $x$ move along $\pi$. With $R=(p_{i_1},\ldots,p_{i_{d+1}})$,
 we have
\[
\sgn(\{x\}\cup R)=\sgn \det \left( \begin{array}{ccccccc}    1 \;      &  \ldots &  1& 1 &   1 \;     & \ldots  & 1  \\
                                                            p_{i_1} &  \ldots & p_{i_{j-1}} &  x & p_{i_j} & \ldots & p_{i_{d+1}} \end{array} \right)
\]
where the position of the column with $x$ is determined by $x$ lying between $p_{i_{j-1}}$ and $p_{i_j}$. Then 
\[
\sgn(\{x\}\cup Q)=\sgn \det  \left( \begin{array}{ccccccc}    1 \;      &  \ldots & 1& 1 &   1 \;     & \ldots  & 1 \;\;  \\
                                                            q_{i_1} &  \ldots &  q_{i_{j-1}}& x & q_{i_j} & \ldots & q_{i_{d+1}} \end{array} \right)
\]
where $Q=(q_{i_1},\ldots,q_{i_{d+1}})$ and the same remark applies to the position of the $x$ column. 

Clearly $\sgn(\{x\}\cup R)=\sgn(\{x\}\cup Q)$ when $x \in P \setminus R$. Thus, it suffices to check how $\sgn(\{x\}\cup Q)$ changes when $x$ moves through $q_i,\ldots,q_j$ for the string $p_i,\ldots,p_j$. Note that $\sgn(\{x\}\cup Q)$ changes only when $x$ passes some point in $Q\cap \pi$.

Just like in the basic case, $\sgn(\{x\}\cup Q)$ does not change when $x$ passes $q_h$ because then $x$ moves from one side of $H^*$ to the other and it also changes places with $q_h$. Thus, $\sgn(\{p_{i-1}\}\cup Q)=\sgn(\{x\}\cup Q)$ when $x$ just passed $q_j$. 

Now we assume that $\al< i$; the other option $\al>i$ is symmetric and follows the same way. There are two cases.

\heading{Case 1:} when $p_j$ and $p_{j+1}$ are on the same side of $H^*$. 
Then $\sgn(\{p_j\}\cup Q)=\sgn(\{p_{j+1}\}\cup Q)$, and so $\sgn(\{p_{i-1}\}\cup Q)=\sgn(\{p_{j+1}\}\cup Q)$, implying $\sgn(\{p_{i-1}\}\cup R)=\sgn(\{p_{j+1}\}\cup R)$. So there is no sign change between $p_{i-1}$ and $p_{j+1}$ in the sign sequence of $R$.

\heading{Case 2:} when $p_j$ and $p_{j+1}$ are on opposite sides of $H^*$. 
Then $H^*\cap p_jp_{j+1}$ is a point $q$, and $\sgn(\{x\}\cup Q)$ 
changes sign when $x$ moves through $q$. Consequently, $\sgn(\{p_{j+1}\}\cup R)=-\sgn(\{p_{i-1}\}\cup R)$, and there is a sign change in the sign sequence 
of $R$ here. 

But since $H^*\cap \pi$ contains already $d+1$ points, Case~2 cannot occur 
anywhere else. Also, the case in Claim~\ref{cor:qqq} cannot come up either,
since that would mean $H^* \cap \pi$ contains $d+2$ points. Thus, the only 
sign change in the sign sequence of $R$ occurs between $p_{i-1}$ and $p_{j+1}$. 
\end{proof}

\section{\boldmath $k$-sequences and flip $k$-sequences}

Now we will define a combinatorial abstraction of
point sequences in $\R^k$. A \emph{$k$-{sequence}}
is a sequence  $S=(a_1,\ldots,a_n)$,
 where $a_1,\ldots,a_n$ are distinct (abstract) elements,
together with a mapping $\sgn$ that assigns either $+1$ or $-1$
to every $(k+1)$-element subset $A\subseteq \{a_1,\ldots,a_n\}$
(sometimes we will regard $A$ as a subsequence, with
the elements in the same order as in $S$).
We will also say that $A$ is \emph{positive} or \emph{negative}
if $\sgn A=1$ or $\sgn A=-1$, respectively.

We subdivide the sequence $S$ into contiguous blocks with one-point overlaps:
The first block is $B_1=(a_1,\ldots,a_{i_1})$ with $i_1$ maximal such that
all $(k+1)$-point subsequences in $B_1$ have the same sign $\si_1$.
The next one is $B_2=(a_{i_1},\ldots,a_{i_2})$ with $i_2$
maximal such that all $(k+1)$-point subsequences in $B_2$ have the
same sign $\si_2$, and so on, up until some $B_m=(a_{i_{m-1}},\ldots,a_n)$,
where $B_m$  either has at most $k$ elements,
 or it has more than $k$ elements and every $(k+1)$-tuple
 in it has the same sign~$\si_m$.

We call this partition the \emph{greedy partition} of $S$;
here both $m=m(S)$ and the blocks $B_j$ are uniquely determined.
Note that each $B_j$, $j < m$, contains a subset
$D_j$ of size $k$ such that $\sgn (\{a_{i_j+1}\}\cup D_j)\ne \si_j$.


The following lemma shows that $S$ has a short subsequence
$S^*$ whose greedy partition is similar to that of~$S$.

\begin{lemma}\label{l:short} There is a subsequence $S^*$ of $S$,
which we call the \emph{reduced version of $S$}, such
that $m(S^*)=m(S)$, every block
of the greedy partition of $S^*$ contains at most $k+3$ elements,
 and the last one exactly $2$.
Moreover, every string of $2k+5$ consecutive elements of $S^*$ contains both
a positive $(k+1)$-tuple and a negative one.
\end{lemma}

\begin{proof}
Let $B_j=(a_{i_{j-1}},\ldots,a_{i_j})$ be a block of the greedy partition
of $S$ with $j<m$. Let us fix a $d$-element subset $D_j$
of $B_j$ as above, i.e., with $\sgn (\{a_{i_j+1}\}\cup D_j)\ne \si_j$.

The subsequence $S^*$ contains the following elements of
 $B_j$:  $a_{i_{j-1}}$,
$a_{i_{j-1}+1}$, $a_{i_j}$, the elements of $D_j$,
and one more (arbitrarily chosen) element if the first three
are all contained in $D_j$. All the other elements are discarded.
From the last block we keep the first two elements.

Let us consider the greedy partition of $S^*$. By induction on $j$,
it is easy to see that for  $j<m$,
the $j$th block $B^*_j$ starts with $a_{i_{j-1}}$,
ends with $a_{i_j}$, and the sign of $D_j\cup \{a_{i_j+1}\}$
is different from $\si_j$, which is the sign of (all)
$(k+1)$-tuples in  $B_j^*$.

It follows that every string of $2k+5$ consecutive elements of $S^*$
contains a full block $B_j^*$ plus the next element $a_{i_j+1}$.
The sign of the first $k+1$ elements of $B_j^*$
is different from $\sgn(D_j\cup \{a_{i_j+1}\})$.
\end{proof}

A $k$-sequence $S=(a_1,\ldots,a_n)$ is called  a \emph{flip
 $k$-sequence} if it has the property as in Lemma~\ref{l:flip};
that is,  for every $k$-element $A \subset \{a_1,\dots,a_n\}$,
the \emph{sign sequence of $A$}
\begin{equation}\label{eq:flipp}
\Bigl(\sgn(\{a_i\}\cup A):i=1,2,\ldots,n,\ a_i\not\in A\Bigr)
\end{equation}
has at most one sign change.
The following result of combinatorial nature is the key
step in the proof of Theorem~\ref{t:pmain}.

\begin{theorem}\label{t:flip}
For every $k\ge 1$ there is $c(k)$ such that the greedy partition of every flip
$k$-sequence has at most $c(k)$ blocks.
\end{theorem}

We prove this result in the next section. Now we show
 how it implies Theorem~\ref{t:pmain}.

\begin{proof}[Proof of Theorem~\ref{t:pmain}.]
We assume that $P=(p_1,\ldots,p_n) \subset \R^d$ is in general position.
Let $\pi=p_1\cdots p_n$ be the corresponding polygonal path.
Lemma~\ref{l:flip} shows that $(p_1,\ldots,p_n)$ with the
sign of $(d+1)$-tuples given by their orientation is a
flip $d$-sequence. Theorem~\ref{t:flip} says that
its greedy partition has at most $c(d)$ blocks.
All $(d+1)$-tuples in $B_j$ have the same sign, so
$B_j=(p_{i_{j-1}},\ldots,p_{i_j})$
is order-type homogeneous, and thus the polygonal path
$p_{i_{j-1}}\cdots p_{j_i}$ is convex.
It follows that $M(d)\le c(d)$.
\end{proof}

\section{Proof of Theorem~\ref{t:flip}}

\begin{proof}
We proceed by induction on $k$.

\heading{The case \boldmath$k=1$. } We will show that $c(1)=3$
(instead of reading this part,
the reader may perhaps prefer to find a simple proof of $c(1)\le 5$, say).

Let $S=(a_1,\ldots,a_n)$ be a flip $1$-sequence, and
let $B_1,\ldots,B_m$ be the blocks of its greedy partition.
Each $B_i$ has the form $(b_i,x_i,\ldots,c_i)$
where $b_{i+1}=c_i$, and $B_i$ contains an element $d_i$
such that $\sgn(d_i,x_{i+1})\ne\si_i$.
Note that $x_1$ and $d_m$ are undefined.
\medskip

\noindent\emph{Observation.}
If  $B_i$ and $B_{i+1}$ are two consecutive blocks, both
positive, then
$d_i$, $c_i=b_{i+1}$, and $x_{i+1}$ are three distinct elements of $S$.
Moreover, for every $a\in S$ preceding $d_i$ we have $(a,x_{i+1})$
 negative, and similarly, for every $a$ following $x_{i+1}$
we have   $(d_i,a)$ negative.
\bigskip

Only the last two statements need an explanation. Since $(c_i,x_{i+1})$
is positive
and $(d_i,x_{i+1})$ is negative, $(a,x_{i+1})$ must be negative
for $a$ preceding $d_i$, for
otherwise, there are two sign changes in the sign sequence of
$\{x_{i+1}\}$. The statement
about $(d_i,a)$ is proved in the same way.

The proof of $c(1)\le 3$ comes in fives steps. We assume w.l.o.g.\ that
$(a_1,a_2)$ is positive.
\begin{enumerate}
\item[\rm Step~1.]
 If all $(a_i,a_{i+1})$ are positive, then $m<4$.
Indeed, supposing $B_4$ exists,
all blocks are positive, $(d_1,x_2)$ is negative,
 and  $(d_1,d_3)$ is negative by the observation.
Also, $(d_3,x_4)$ is negative and there are two sign changes in the sign sequence of $\{d_3\}$, since
 $(b_3,d_3)$ or $(d_3,c_3)$ (or both) are positive.

\item[Step 2.]
 If $j$ is the smallest index with $(a_j,a_{j+1})$ negative, then $a_j=c_i=b_{i+1}$,
$B_i$ is a positive block, and  $B_{i+1}$ is
a negative one. Assume $B_{i-1}$ exists. Then it is positive,
$(d_{i-1},x_i)$ is negative, and thus $(d_{i-1},a_j)$ is negative by the
observation. But then there are two
sign changes in the sign sequence of $\{a_j\}$:
 $(d_{i-1},a_j)$ and $(a_j,x_{i+1})$ are negative and
$(b_i,a_j)$ is positive. Thus $B_{i-1}$ cannot exist,
$i=1$, and there is a single block before~$a_j$.

\item[Step 3.] Thus $B_1$ is positive and $B_2$ negative.
Assume $B_3$ negative; then
$(d_2,x_3)$ is positive and so is $(b_2,x_3)$ by the observation. Consequently,
 there are two
sign changes in the sign sequence of $\{b_2\}$:
$(b_1,b_2)$ and $(b_2,x_3)$ are positive but $(b_2,x_2)$
is negative. We conclude that $B_3$ is a positive block.

\item[Step 4.] Assume $B_4$ exists and is positive. Then $(d_3,x_4)$ is negative and so is $(b_3,x_4)$
by the observation. Then there are two sign changes in the sign sequence of
$\{b_3\}$: $(b_2,b_3)$ and
$(b_3,x_4)$ are negative and $(b_3,c_3)$ is positive.

\item[Step 5.] We are left with the case when $B_1,B_3$ are positive and $B_2,B_4$ negative. If $(b_2,b_4)$
is positive, then there are two sign changes in the sign sequence of
$\{b_2\}$: $(b_1,b_2)$ and $(b_2,b_4)$
are positive and $(b_2,c_2)$ negative. Similarly, if $(b_2,b_4)$ is negative, then there are two sign
changes in the sign sequence of~$\{b_4\}$.

Consequently, $B_4$ does not exist: $m<4$ and so $c(1)\le 3$.
\end{enumerate}

The example $S=(a_1,a_2,a_3,a_4)$
with $a_1,a_2$ and $a_3,a_4$ positive and all other pairs negative shows that $c(1)=3$.

\heading{The inductive step from $k-1$ to $k$. }
Assuming that the greedy partition of each flip $(k-1)$-sequence
has at most $c(k-1)$ blocks, we will show that the greedy partition
of an arbitrary flip $k$-sequence $S=(a_1,\ldots,a_n)$
has at most $c(k):= 1+ (4k+10)c(k-1)/k$ blocks.

So we suppose the contrary that $S$ as above has $m>c(k)$ blocks.
We can further assume that $S$ is reduced in the sense of
Lemma~\ref{l:short}, for otherwise, we can replace $S$ by $S^*$.
Since each $B_i$, $i<m$, has at least $k+1$ elements, and $|B_m| = 2$,
the length of $S$ is at least
\[
n\ge (m-1)k+2 > (4k+10)c(k-1) +2.
\]

We consider the sequence
$T=(a_1,\ldots,a_{n-1})$ and regard it as a $(k-1)$-sequence
by defining, for a $k$-element
$A\subset \{a_1,\ldots,a_{n-1}\}$, the sign $\sgn A:=\sgn(A\cup\{a_n\})$.
It is clear that $T$ is a flip $(k-1)$-sequence, and so
its greedy partition has at most $c(k-1)$ blocks.
One of the blocks, which we call $B$, has at least
$(n-1)/c(k-1)\ge 4k+10$ elements. We may assume w.l.o.g.\ that
$\sgn A=+1$ for every $k$-element subset of~$B$.

Since $S$ is reduced,
there is a positive $(k+1)$-tuple $(b_1,\ldots,b_{k+1})$ among the first
$2k+5$ elements of $B$, and
a negative $(k+1)$-tuple $(b_{k+2},\ldots,b_{2k+2})$
among the last $2k+5$ elements of $B$. The sign of the $(k+1)$-tuple
$(b_i,\ldots,b_{i+k})$ changes from $+1$ to $-1$ as $i$ moves
through $1,2,\ldots,k+2$, and so there is some
$j$ with $\sgn (b_j,\ldots,b_{j+k+1})=+1$ and
$\sgn (b_{j+1},\ldots,b_{j+k+2})=-1$.

We set $A=\{b_{j+1},\ldots,b_{j+k+1}\}$.
Then we have $\sgn (\{b_j\} \cup A)=+1$ and $\sgn (A\cup \{b_{j+k+2}\})=-1$,
while $\sgn (A \cup \{a_n\})=+1$ by the choice of the
block $B$. Hence the sign sequence of $A$ has at least two sign
changes, contradicting the assumption that $S$
is a flip $k$-sequence. This contradiction finishes the proof
of Theorem~\ref{t:flip}.
\end{proof}

\heading{Remark.}
This argument gives $c(k)=\exp(O(k))$. We note that $c(1)=3$ and $M(1)=3$.
The above proof gives $c(2)\le 22$ while $M(2)=4$.

\section{From polygonal paths to curves: proof of Theorem~\ref{t:main}}\label{s:poly-c}

Here we show how Theorem~\ref{t:pmain} implies Theorem~\ref{t:main}.

We assume  $\gamma\:I\to\R^d$ is a {\crossi{d+1}} curve.

Let us say that an $n$-tuple $T=(t_1,\ldots,t_n)$,
$t_1,\ldots,t_n\in I$, $t_1<\cdots<t_n$, is an \emph{$\eps$-sample} if
every subinterval of $I$ of length $\eps$ contains some $t_i$.
Let $\pi=\pi(\gamma,T)=\gamma(t_0)\gamma(t_1)\cdots\gamma(t_n)$ be the polygonal
line determined by~$T$.

First we observe that for every $\eps>0$, there is an $\eps$-sample
$T$ with $\pi(\gamma,T)$ in general position. Indeed, having already placed
$k$ points of $T$, so that their $\gamma$-images are in general
position, we consider the finitely many hyperplanes spanned by $d$-tuples of
these $\gamma$-images. Since $\gamma$ is {\crossi{d+1}}, each of
these hyperplanes contains at most one extra point of $\gamma$, and
so at every step of the construction, we have only finitely many
excluded points of $I$. Thus, we can construct an $\eps$-sample
as desired.

Next, for every $\eps>0$, we fix an $\eps$-sample $T=T(\eps)$
with $\pi(\gamma,T(\eps))$ in general position.
Let $M=M(d)$ be as in Theorem~\ref{t:pmain}; by that
theorem, we can also fix a subdivision
of $I$ into $M$ subintervals such that the restriction of
$\pi(T(\eps),\gamma)$ on each of them is convex.
By compactness, these subdivisions have a cluster point for $\eps\to 0$;
we denote its intervals by $I_1,\ldots,I_M$.

It remains to show that $\gamma$ restricted to each $I_j$ is convex.
This follows from the next lemma, applied with $I=I_j$ and $\gamma=\gamma_j$.

\begin{lemma}\label{l:sample}
 Let $\gamma\:I\to\R^d$ be a {\crossi{d+1}} curve, and let us suppose that
for every $\eps>0$ there is an $\eps$-sample $T(\eps)$ such that the
corresponding polygonal path $\pi(\gamma,T(\eps))$ is in general position
and convex. Then $\gamma$ is convex as well.
\end{lemma}

\begin{proof} For contradiction, we suppose that there is a hyperplane
$h$ intersecting $\gamma$ in at least $d+1$ points.

First we observe that these points can be assumed to span $h$:
if their affine hull $F$ had dimension smaller than $d-1$,
then since $\gamma\not\subset F$, we could rotate $h$ around $F$
and thus get more than $d+1$ intersections.

Let us say that a point $\gamma(t)\in h$, $t\in I$, is a
\emph{generic intersection} with $h$
if for an arbitrarily small neighborhood $U$ of $t$,
$\gamma(U)$ intersects both of the open halfspaces bounded by $h$
(as usual, we count generic intersections with multiplicity,
so the generic intersection is actually determined by $t$).
We claim that there is a hyperplane $h'$ with at least $d+1$
generic intersections.

For easier description, let us imagine $h$ horizontal.
An intersection that is not generic is either an endpoint of
$\gamma$, or it is a point $p$ where $\gamma$ touches $h$,
with a sufficiently small open neighborhood of $p$ on $\gamma$
lying all strictly above $h$ or all strictly below it;
let us call such intersections \emph{top-touching}
or \emph{bottom-touching}.

Let $q_1,q_2,\ldots,q_k$ be the non-generic intersections
of $\gamma$ with $h$. At least $k-1$ of these are affinely independent,
say $q_1,\ldots,q_{k-1}$,
and thus we can make an arbitrarily small movement of $h$
so that a prescribed subset of $\{q_1,\ldots,q_{k-1}\}$
ends up below $h$ and its complement above~$h$. The previously generic
intersections remain generic, provided that the movement was sufficiently
small.

Now if $q_i$ was bottom-touching and it lies above $h$ after the move,
then it yields (at least) two generic intersections with $h$,
and similarly for top-touching. If $q_i$ is an endpoint, then
it yields at least one generic intersection, provided that
$h$ was moved in the right direction.

Hence by an appropriate move we can always get
at least $d+1-k+2(k-3)+2=d+k-3$ generic intersections,
which is at least $d+1$ for $k\ge 4$. So it remains to discuss
the cases $1\le k\le 3$.

For $k\le 2$, the non-generic intersections are distinct and
thus affinely independent, and so we can get $k$ new generic
intersections by a suitable move. For $k=3$, there are two
affinely independent non-generic intersections, at least one
of them top-touching or bottom-touching, and hence
we can also get $3$ new generic intersections by a suitable move.
Thus, we have obtained a hyperplane $h'$ with at least
$d+1$ generic intersections as required.

Let $t_1,\ldots,t_{d+1}\in I$, $t_1<\cdots<t_{d+1}$, be the parameter
values corresponding to these generic intersections with $h'$.
To finish the proof of the lemma, we fix a sufficiently small
$\eps>0$ and intervals $J_1^+,J_1^-,\ldots,J_{d+1}^+,
J_{d+1}^-\subset I$, each of length at least $\eps$, such that
$J_i^+$ and $J_i^-$ are in a small neighborhood of $t_i$
(and thus they lie left of $J_{i+1}^+\cup J_{i+1}^-$),
and $\gamma(J_i^+)$ lies above $h'$ and $\gamma(J_i^-)$ below it.

Suppose that $J_1^+$ precedes $J_1^-$, for example.
Then we choose points $u_0,u_1,\ldots,u_{d+2}\in T(\eps)$
with $u_0\in J_1^+$, $u_1\in J_1^-$, $u_2\in J_2^+$,
$u_3\in J_3^-$, $u_4\in J_4^+$, etc. Then the polygonal
line $\pi(\gamma,T(\eps))$ changes sides of $h'$ at least $d+1$ times,
and thus it has at least $d+1$ intersections with $h'$.
Since the position of $h'$ is generic, this shows that
$\pi(\gamma,T(\eps))$ is not convex---a contradiction proving the
lemma, and also concluding the proof of Theorem~\ref{t:main}.
\end{proof}

\section{The lower bound for order-type homogeneous subsequences}\label{s:OT}

\heading{Super-order type homogeneity. }
The following strengthening of order-type homogeneity was considered
in \cite{EMRS}: a point sequence $P=(p_1,p_2,\ldots,p_n)$ in $\R^d$ is
\emph{super-order type homogeneous} if, for every $k=1,2,\ldots,d$,
the projection of $P$ to the first $k$ coordinates is order-type
homogeneous (this includes the assumption that all of these projections
are in general position---let us abbreviate this by saying that $P$
is in \emph{super-general position}).

It is easily seen, e.g.,
 by Ramsey's theorem, that for every $d$ and $n$ there is $N$ such
that every $N$-point sequence in super-general position in $\R^d$
contains a super-order type homogeneous subsequence of length $n$.
Let us denote the corresponding Ramsey function by $\COT_d(n)$.

It was shown in \cite{EMRS} that $\COT_d(n)\ge \twr_d(n-d)$.
Thus, to prove Theorem~\ref{t:OT4}, the lower bound for $\OT_d$,
and having Theorem~\ref{t:pmain} at our disposal,
it suffices to verify the following.

\begin{lemma} For all $d\ge 2$,
$\OT_{d}(n)\ge \COT_{d}(\Omega(n))$.
\end{lemma}

\begin{proof} Given $n$, let us set $N=\OT_{d}(n)$, and consider
an $N$-point sequence in super-general position in $\R^{d}$.
By definition, it contains an $n$-point order-type homogeneous
subsequence $P_1$.

By Lemma~\ref{f:act}, the polygonal path given by $P_1$ is
convex, i.e., \crossi{d}, and hence its projection onto the first $d-1$
coordinates is {\crossi d} as well. So by the assumption,
it can be subdivided into at most $M(d-1)$
polygonal paths that are {\crossi{d-1}}.
One of them corresponds, by Lemma~\ref{f:act} again,
to a subsequence $P_2$ of $P_1$ of length at least $n/M(d-1)$
whose projection to the first $d-1$ coordinates is order-type
homogeneous.

Analogously we construct $P_3,\ldots,P_{d}$, where
$|P_i|\ge |P_{i-1}|/M(d-i+1)$ and the projections of $P_i$
to the first $k$ coordinates, for $k=d-i+1,d-i+2,\ldots,d$,
are order-type homogeneous. In particular, $P_{d}$ is
the desired super-order type homogeneous subsequence of
length~$\Omega(n)$.
\end{proof}

\subsection*{Acknowledgment}
We would like to thank Marek Eli\'a\v{s}, Edgardo Rold\'an-Pensado,
Zuzana Safernov\'a, and Pavel Valtr for inspiring discussion
at initial stages of this research, R.\,A.~Zalik for
answering our question about Chebyshev systems, and
F.~Santos, B.~Shapiro, and O.\,R.~Musin for useful comments and references.
The first and the second author were supported by the ERC
Advanced Research Grant no.~267165 (DISCONV), and the first author by the
Hungarian National Research Grant K~83767 as well.

\bibliographystyle{alpha}
\bibliography{cg,geom}

\end{document}